\documentclass[aap,MSNbibl,citesort,seceqn,dvips]{arximspdf}
\usepackage{graphicx}

\doi{10.1214/11-AAP765}
\volume{21}
\issue{6}
\pubyear{2011}
\firstpage{2053}
\lastpage{2074}

\makeatletter

\newtheorem{theorem}{Theorem}[section]
\newtheorem{lemma}[theorem]{Lemma}

\newproclaim{example}[theorem]{Example}
\newproclaim{definition}[theorem]{Definition}
\newproclaim{remark}[theorem]{Remark}

\newcommand{\prt}{\partial}
\newcommand{\eps}{\varepsilon}
\newcommand{\PP}{\mathbb{P}}
\newcommand{\E}{\mathbb{E}}
\newcommand{\R}{\mathbb{R}}
\newcommand{\T}{{\mathcal T}}
\newcommand{\M}{{\mathcal M}}
\newcommand{\K}{{\mathcal K}}

\newcommand{\EE}{{\mathcal E}}
\newcommand{\FF}{{\mathcal F}}

\newcommand{\bn}{\mathbf{n}}
\newcommand{\bX}{\mathbf{X}}
\newcommand{\dist}{\operatorname{dist}}

\newcommand{\bv}{\mathbf{v}}
\newcommand{\bu}{\mathbf{u}}
\newcommand{\by}{\mathbf{y}}
\newcommand{\bz}{\mathbf{z}}
\newcommand{\wt}{\widetilde}

\newcommand{\bD}{\mathbf{D}}

\newcommand{\ol}{\overline}

\makeatother

\begin{document}
\begin{frontmatter}

\title{Archimedes' principle for Brownian liquid\thanksref{T1}}
\runtitle{Archimedes' principle}

\begin{aug}
\author[A]{\fnms{Krzysztof} \snm{Burdzy}\corref{}\ead[label=e1]{burdzy@math.washington.edu}},
\author[A]{\fnms{Zhen-Qing} \snm{Chen}\ead[label=e2]{zchen@math.washington.edu}} and
\author[A]{\fnms{Soumik} \snm{Pal}\ead[label=e3]{soumik@math.washington.edu}}
\runauthor{K. Burdzy, Z.-Q. Chen and P. Pal}
\affiliation{University of Washington}
\address[A]{Department of Mathematics\\
University of Washington\\
Box 354350\\
Seattle, Washington 98195\\
USA\\
\printead{e1}\\
\hphantom{E-mail: }\printead*{e2}\\
\hphantom{E-mail: }\printead*{e3}}
\end{aug}

\thankstext{T1}{Supported in part by NSF Grants DMS-09-06743,
DMS-10-07563 and
by Grant N N201 397137, MNiSW, Poland.}

\received{\smonth{5} \syear{2010}}
\revised{\smonth{1} \syear{2011}}

%
\begin{abstract}
We consider a family of hard core objects moving as independent
Brownian motions confined to a vessel by reflection. These are
subject to gravitational forces modeled by drifts. The stationary
distribution for the process has many interesting implications,
including an illustration of the Archimedes' principle. The analysis
rests on constructing reflecting Brownian motion with drift in a
general open connected domain and studying its stationary
distribution. In dimension two we utilize known results about sphere
packing.
\end{abstract}

%
\begin{keyword}[class=AMS]
\kwd{60J65}
\kwd{60K35}.
\end{keyword}
\begin{keyword}
\kwd{Archimedes' principle}
\kwd{reflecting Brownian motion with drift}
\kwd{stationary distribution}
\kwd{sphere packing}
\kwd{surface of liquid}
\kwd{centrifuge effect}.
\end{keyword}

\end{frontmatter}

\section{Introduction}\label{intro}

We consider a model involving ``hard core objects'' (typically,
spheres) moving as independent Brownian motions, reflecting
from each other and subjected to a constant ``force,'' that
is, having a constant drift. The objects are confined to a
``vessel'' by reflection, that is, they cannot leave a subset
of Euclidean space.
Our ``toy model'' illustrates several well-known physical phenomena for
liquids, under some technical (mathematical) assumptions.
We prove some theorems for moving objects of any size and
shape but the most interesting examples involve a large number
of spheres, most of them small.

The first of the three phenomena that our model generates is
tight packing of the objects under large pressure and the
formation of the surface of the liquid, that is, a hyperplane
such that most spheres are below the surface and there is
little room to pack any more spheres below the surface.

The second phenomenon is the ``centrifuge'' effect. Centrifuges
are used to separate materials consisting of small particles
(molecules) with different mass. In this example, we consider
spheres of the same size but subject to different ``forces,''
that is, drifts. The ``heavy'' spheres tend to be closer to the
bottom than light spheres.

The third phenomenon is Archimedes' principle which says that
an object immersed in a fluid is buoyed up by a force
equal to the weight of the fluid displaced by the object.
We will illustrate this principle by a family of two-dimensional discs.
One disc is large and it is
submerged in a ``liquid'' consisting of a very large number of
much smaller discs. The large disc either ``floats'' or
``sinks'' depending on the ratio of its drift and the drift of
small discs. This model is limited to the two-dimensional case
because the classical sphere packing problem is completely
understood only in this case. A similar probabilistic theorem
can be stated and proved in higher dimensions if the relevant
information on sphere packing is available.

Finally, we will give an example involving objects with
``inertia,'' in which high inertia will make the large object
(disc) sink more easily even if the drift of this object would
not be sufficient to make it sink without inertia. The inertia
is modeled by oblique reflection and low diffusion constant.
Two spheres are said to reflect in an oblique way if the amount of
push (a multiple of the local time) experienced by the spheres
during reflection is not identical for the two spheres.

Although Section \ref{sec:stat} contains technical material needed for
our main results,
it may have some independent interest. There we construct reflected Brownian
motion with drift in an arbitrary Euclidean domain and find its
stationary distribution. The reason for this great level of generality
is that we apply these results to the configuration space of hard core
objects. Even if the objects are spheres, the configuration space does
not have to be smooth. Some regularity properties of the configuration
space for nonoverlapping balls were proved in \cite{DLM}, Proposition 4.1.

The present article has its roots in an analogous
one-dimensional model studied briefly in Section 2 of
\cite{BPS}. A construction of an infinite system of reflecting
Brownian hard core spheres was given in \cite{O}. See the
\hyperref[intro]{Introduction} and references in that paper for the history and
ramifications of the problem.

We are grateful to Charles Radin for the following remarks and
references (but we take the responsibility for any inaccuracies). A
physical system that could reasonably be called a Brownian fluid is a
colloid, like milk. Since we are using ``reflecting hard spheres,''
this specializes to noncohesive, hard-particle colloids. One of the
classic material properties which are demonstrated in colloids is the
fluid/solid phase transition known by simulation in the hard sphere
model. Although this was first demonstrated earlier by others, the
definitive paper seems to be \cite{RDXRC}. A short expository
introduction to related problems can be found online~\cite{R}. A recent
preprint concerned with the motion of globules is \cite{F}.

The rest of the paper is organized as follows. Section
\ref{sec:stat} is devoted to the analysis of reflected Brownian
motion with drift in an arbitrary Euclidean domain.
Section~\ref{sec:conf} contains some lemmas about the geometry of the
configuration space of objects in a vessel. We give a
sufficient condition for the existence of a stationary
distribution for a family of reflecting objects in Section
\ref{sec:exist}. Finally, Section \ref{sec:examples} contains
our main results, informally discussed above. Our mathematical
model is formally introduced at the beginning of Sections
\ref{sec:conf} and \ref{sec:exist}.

\section{Stationary distribution for reflected Brownian motion with
drift}\label{sec:stat}

Constructing reflected Brownian motion with a constant drift
$\bv$ in a \textit{general} open connected set $D \subset\R^n$ is
quite delicate. Even its definition needs a careful
formulation. This will be done in this section. We will then
give a formula for the stationary distribution of this process
(see \cite{DM,DR,HW} for some related results). A noteworthy
aspect of Theorem \ref{th:j14.2} below is that we do not
require any regularity assumptions on the boundary of $D$.

When $D$ is a $C^3$-smooth domain in $\R^n$, a (normally) reflecting
Brownian motion $X$ with constant drift $\bv\in\R^n$ can be described
by the following SDE:
%
\begin{equation}\label{e:2.1}
dX_t = dB_t + \bv \,dt + {\bn} (X_t) \,dL_t, \qquad  t\geq0,
\end{equation}
with the constraint that $X_t\in\overline D$ and $L$ is a
continuous nondecreasing process that increases only when $X$
is on the boundary $\partial D$. In (\ref{e:2.1}), $B$ is
Brownian motion on $\R^n$ and ${\bn}$ is the unit inward normal
vector field of $D$ on~$\partial D$. When~$D$ is $C^3$-smooth,
the strong existence and pathwise uniqueness of solution to
(\ref{e:2.1}) is guaranteed by~\cite{LS}. When $D$ is a
Lipschitz domain and $\bv=0$,~(\ref{e:2.1}) has a unique weak
solution by \cite{BBC}, Theorem 1.1(i). Using a~%
Girsanov transform, we conclude that weak existence and weak uniqueness
hold for~(\ref{e:2.1}) with nonzero constant drift $\bv$. When
$\bv=0$, this equation has a unique strong solution when the
domain $D$ is $C^{\gamma}$ with $\gamma> 3/2$,
by~\cite{BB}, Theorem~1.1. To motivate the definition of reflecting
Brownian motion with constant drift $\bv$ in a general,
possibly nonsmooth, domain, observe that when $D$ is~$C^3$ and
$\PP_x$ denotes the law of the solution $X$ to (\ref{e:2.1})
with $X_0=x$, then $\{ X, \PP_x, x\in\overline D\}$ forms a
time-homogeneous strong Markov process with state space
$\overline D$. Let $\{P_t, t\geq0\}$ denote its transition
semigroup and $\rho(x):= e^{\bv\cdot x}$. It is easy to check
that for every $t>0$, $P_t$ is a symmetric contraction operator
on $L^2(\overline D, m)$, where $m(dx):= 1_D(x) \rho(x)^2 \,dx$.
Let $(\EE, \FF)$ denote the Dirichlet form of $X$ on
$L^2(\overline D, m)$; that is,
%
\begin{eqnarray}\label{e:2.2}
\FF&=&\biggl\{ u\in L^2(D; m)\dvtx \lim_{t\to0} \frac1t (u-P_tu,
u)_{L^2(D; m)}
<\infty\biggr\}, \\ \label{e:2.3}
\EE(u, v) &=& \lim_{t\to0} \frac1t (u-P_tu, v)_{L^2(D; m)}
\qquad\mbox{for } u, v\in\FF.
\end{eqnarray}
Then it is easy to check [see also the proof of Theorem
\ref{th:j14.2}(i) below] that
%
\begin{eqnarray} \label{e:2.4}
\FF&=& \{u\in L^2(D; m)\dvtx \nabla u \in L^2(D; m) \} , \\
\label{e:2.5}
\EE(u, v) &=&\frac12 \int_D \nabla u(x) \cdot\nabla v(x) m(dx)
\qquad \mbox{for } u, v\in\FF
\end{eqnarray}
and that\vspace*{1pt} $(\EE, \FF)$ is a regular Dirichlet form on $L^2(\overline
D; m)$
in the sense that $C_c(\overline D)\cap\FF$ is dense both in
$C_c(\overline D)$
with respect to the uniform norm and in~$\FF$ with respect to the
Hilbert norm
$\sqrt{\EE_1 (u , u )}:=\sqrt{ \EE(u , u )+ (u , u )_{L^2(D; m)}}$.
Here $C_c(\overline D)$ is the space of continuous functions on
$\overline D$
with compact support. This motivates the following definition.
\begin{definition}\label{D:2.1}
Let $D\subset\R^n$ be an open connected set. A continuous
Markov process $\{X_t, \PP_x, x\in D\}$ taking values in
$\overline D$ will be called a reflecting Brownian motion with constant
drift $\bv$ if
it is symmetric with respect to the
measure $m(dx):= 1_D(x) e^{2 \bv\cdot x} \,dx$, and
$\PP_x(X_t\in\partial D)=0$ for every $x\in D$ and $t>0$ and
whose associated Dirichlet form $(\EE, \FF)$ in the sense of (\ref
{e:2.2}) and~(\ref{e:2.3})
is given by (\ref{e:2.4}) and (\ref{e:2.5}).
\end{definition}
\begin{remark}\label{R:2.2}
\begin{longlist}
\item
Note that reflecting Brownian motion
$X$ with constant drift $\bv$
does not have to be a strong Markov process on $\overline D$.
Nevertheless, its associated transition semigroup $\{P_t, t\geq0\}$ is well
defined by the formula
\[
\int_D f(x) P_t g (x) m(dx)= \E_m [ f(X_0) g(X_t)]
\qquad\mbox{for every } f, g \in C_c(D).
\]
Thus, defined
$\{P_t, t\geq0\}$ is a strongly continuous
symmetric semigroup in $L^2(D; m)$
and so the Dirichlet form $(\EE, \FF)$ is well defined (see \cite{FOT},
Table 1, pa\-ge 18 and Theorems 1.3.1 and 1.4.1).

\item When $\bv=0$, the Dirichlet form $(\EE, \FF)$ of
(\ref{e:2.4}) and (\ref{e:2.5})
will be denoted as $(\EE^0, W^{1,2}(D))$. Note that $W^{1,2}(D)$ is the
classical
Sobolev space on $D$ of order $(1,2)$.
\end{longlist}
\end{remark}
\begin{theorem}\label{th:j14.2} Suppose that $\bv\in\R^n$ and
$D \subset\R^n$ is open and connected.

\begin{longlist}
\item
There exists a unique in law reflected Brownian motion
$X_t$ in $D$ with constant drift $\bv$.
\item
Suppose that $D_k$, $k\geq1$, is a sequence of open
connected sets with smooth boundaries such that $D_k \subset
D_{k+1}$ for $k\geq1$ and $\bigcup_{k\geq1} D_k = D$. Let
$X^k_t$ be the reflected Brownian motion in $D_k$ with a
constant drift $\bv$. Assume that $X^k_0 \to x_0\in D$ in
distribution. Then $\{X^k_t, t\geq0\}$ converge weakly to
$\{X_t, t\geq0\}$, $X_0 = x_0$, in $C([0,\infty); \overline
D)$, as $k\to\infty$.
\item
The function $f(x) = \exp( 2 x\cdot\bv) $, $x\in D
$, is the density of an invariant measure for $X$. If this
density is integrable over $D$, then $\mu(dx):=f(x)\,dx/\int_{D}
f(x)\,dx$ is the unique stationary distribution for $X$.
\item
Suppose that $D$ and vector $\bv_1 \ne0$ are fixed
and let $\bv_b = b \bv_1$ for $b>0$.
Assume that $\int_D \exp( 2 x\cdot\bv_1) \,dx<\infty$.
Then there exists a
stationary distribution ${\mu}_b(d x)$ for reflected Brownian
motion in $D$ with drift $\bv_b$, for every $b\geq1$. Suppose
that $c_0 := \sup_{x\in D} (2 x\cdot\bv_1) < \infty$ and let $
D(\eps) = \{x\in D\dvtx 2x \cdot\bv_1 > c_0 - \eps\} $ for
$\eps>0$. Then for every $\eps>0$ we have $\lim_{b\to\infty}
\mu_b (D(\eps)) = 1$.\vadjust{\goodbreak}
\end{longlist}
\end{theorem}
\begin{pf}
(i) The following facts have been established in \cite{C1,C2}.
Let $(\EE^0, \FF)$ be the Dirichlet form defined by
(\ref{e:2.4}) and (\ref{e:2.5}) with constant function 1 in place
of $\rho$.
There is a continuous strong Markov process
$\{Y^*, \PP^*_x, x\in D^*\}$ on the Martin--Kuramochi compactification
$D^*$ of $D$ whose associated Dirichlet space is $(\EE^0,
W^{1,2}(D))$.
For $1\leq i\leq m$, let $f_i(x):= x_i$. Then
$f_i$ admits a~quasi-continuous extension to $D^*$. Define
$Y_t= (f_1(Y_t^*),\ldots, f_m (Y_t^*))$, which we call\vspace*{1pt}
symmetric reflecting Brownian motion on $\overline D$. Process
$Y$ admits the following decomposition:
\[
Y_t=Y_0+B_t +N_t, \qquad  t\geq0,
\]
where $B$ is Brownian motion on $\R^n$ and $N_t$ is an $\R^n$-valued
process locally
of zero quadratic variation.

We now construct reflecting Brownian motion $X$ on $\overline
D$ with constant drift~$\bv$ through Girsanov transform. Let
$\{\FF_t, t\geq0\}$ be the minimal augmented filtration
generated by $Y^*$. (This notation is similar to the one used
for the domain $\FF$ of the Dirichlet form but there will be
little opportunity for confusion.) For $x\in D^*$, define the
measure $\PP_x$ by
\[
\frac{d \PP_x}{d \PP_x^*} = M_t:=\exp\biggl( \int_0^t \bv  \,dB_s -
\frac
12 \int_0^t
| \bv|^2 \,ds \biggr)  \qquad\mbox{on each } \FF_t.
\]
Since the right-hand side forms a martingale under $\PP_x^*$,
$Y^*$ under $\PP_x$ has infinite lifetime. By the same (but
simpler) argument as that for \cite{CFTYZ}, Lemma~2.4 [with
$\rho(x):=e^{\bv\cdot{\mathbf x}}$ there], one can show that $(Y^*,
\PP_x, x\in D^*)$ is a symmetric Markov process with respect to
the measure $m(dx)=1_D(x) \rho(x)^2 \,dx$. On the other hand, as
a special case of \cite{CFKZ}, Theorem 3.1, the asymmetric
Dirichlet form $(\EE^*, \FF^*)$ associated with $\{Y^*, \PP_x\}$
in $L^2(D^*; dx)$ is $(\EE^*, W^{1,2}(D))$, where
\[
\EE^* (u, v)= \frac12 \int_D \nabla u(x) \cdot\nabla v(x) \,dx -
\int_D \bigl(\bv\cdot\nabla u(x)\bigr) v(x) \,dx, \qquad  u, v\in W^{1,2}(D).
\]
Denote by $\{G_\alpha, \alpha>0\}$ the resolvent of $\{Y^*,
\PP_x\}$. The above means that (cf.~\cite{MR}, Theorem I.2.13)
\[
\Bigl\{ u\in L^2 (D; dx)\dvtx \sup_{\beta>0} (u-\beta G_\beta u,
u)_{L^2(D;dx)}<\infty
\Bigr\} = W^{1,2}(D)
\]
and for $u, v\in W^{1,2}(D)$,
%
\begin{eqnarray}\label{e:2.6a}
&&\lim_{\beta\to\infty} (u-\beta G_\beta u,
v)_{L^2(D;dx)}\nonumber\\[-8pt]\\[-8pt]
&&\qquad= \frac12 \int_D \nabla u(x) \cdot\nabla v(x) \,dx -
\int_D \bigl(\bv\cdot\nabla u(x)\bigr) v(x) \,dx .\nonumber
\end{eqnarray}
Let $(\overline\EE, \overline\FF)$ be the symmetric
Dirichlet
form of $\{Y^*, \PP_x\}$ in $L^2(D^*; m)$. Denote by
$bW^{1,2}_{c}(D)$ and $b\FF_c$ the families of bounded
functions in $W^{1,2}(D)$ and in~$\FF$ with compact support,
respectively. Note that $bW^{1,2}_{c}(D)=b\FF_c$ and it is a~%
dense linear subspace in both $(W^{1,2}(D), (\EE^0_1)^{1/2})$
and $(\FF, \EE_1^{1/2})$. By~(\ref{e:2.6a}), for $u\in
bW^{1,2}_{c}(D)=b\FF_c$,
\begin{eqnarray*}
&&
\lim_{\beta\to\infty} (u-\beta G_\beta u, u)_{L^2(D; \rho^2 \,dx)}\\
&&\qquad= \lim_{\beta\to\infty} (u-\beta G_\beta u, \rho^2 u)_{L^2(D; dx)}
\\
&&\qquad= \frac12 \int_D \nabla u(x) \cdot\nabla(\rho^2 u)(x) \,dx -
\int_D \bigl(\bv\cdot\nabla u(x)\bigr) (\rho^2 u) (x) \,dx \\
&&\qquad= \frac12 \int_D |\nabla u(x)|^2 m(dx).
\end{eqnarray*}
This implies
(cf. \cite{FOT}, Lemma 1.3.4(ii))
that $b\FF_c \subset\overline\FF$ and
for $u, v\in b\FF_c$,
\[
\overline\EE(u, v) = \frac12 \int_D \nabla u(x)\cdot\nabla v(x)
m(dx)=\EE(u, v).
\]
It follows
that
$\FF\subset\overline\FF$ and $\overline\EE= \EE$ on $\FF$.
Conversely, for $u\in b \overline\FF_c$,
we have
$\rho^{-2}u \in b \overline\FF_c$.
Hence,
\begin{eqnarray*}
\lim_{\beta>0} (u-\beta G_\beta u, u)_{L^2(D; dx)}
&=& \lim_{\beta>0} (u-\beta G_\beta u, \rho^{-2}u)_{L^2(D; m)}\\
&=&\overline\EE(u, u)^{1/2}   \overline\EE(\rho^{-2}u,   \rho
^{-2}u)^{1/2}
<\infty.
\end{eqnarray*}
This implies that $u\in bW^{1,2}_c(D)\subset\FF$. In other
words, $b\overline\FF_c \subset\FF$. We conclude that
$(\overline\EE, \overline\FF)=(\EE, \FF)$. This completes the
proof that $(Y, \PP_x)$ is a reflecting Brownian motion with
constant drift on $\overline D$. To emphasize that we have
constructed reflecting Brownian motion with drift as in
Definition \ref{D:2.1}, we switch to our original notation used
in that definition, that is, processes $Y^*$ and $Y$ under
measure $\PP_x$ will be denoted $X^*$ and $X$, respectively.

Next we establish uniqueness. Suppose that $\wt X$ is another
reflecting Brownian motion with constant drift $\bv$ on $D$. By
\cite{FOT}, Lemma 1.3.2 and Theorem~1.3.1, the transition
semigroup of $\wt X$ should be the same as the transition
semigroup of $X$. So as continuous processes, $\wt X$ and $X$
share the same law under the initial distribution~$m$. Since
the subprocesses of $\wt X$ and $X$ killed upon leaving $D$ are
Brownian motions in $D$ with constant drift $\bv$, it follows
that $\wt X$ and $X$ have the same distribution for every
starting point $x\in D$.

(ii) Since $D_k$ has smooth boundary, the Martin--Kuramochi
compactification of $D_k$ coincides with the Euclidean closure
of $D_k$. So reflecting Brownian motion $\{Y^k, \PP^{*, k}_x\}$
on $\overline D_k$ is a strong conservative Markov process with
continuous sample paths. Each $Y^k$ admits a Skorokhod
decomposition (cf.~\cite{C2})
\[
Y^k=Y^k_0+B^k_t + \int_0^t {\bn}_k (Y^k_s) \,dL^k_s, \qquad  t\geq0,
\]
where $B^k$ is Brownian motion on $\R^n$, $\bn_k (x)$ is the unit
inward normal vector
at $x\in\partial D_k$ and $L^k$ is the boundary local time for
reflecting Brownian
motion $Y^k$.
As we saw in (i) above, each $X^k$ can be
generated from reflecting Brownian motion $Y^{k}$ on $\overline D_k$ by the
Girsanov transform
\[
\frac{d \PP^k_x}{d \PP_x^*} = M^k_t:=\exp\biggl( \int_0^t \bv
\,dB^k_s -
\frac12 \int_0^t
| \bv|^2 \,ds \biggr) \qquad \mbox{on each } \FF^k_t.
\]
Since $Y^k_0=X^k_0$ is assumed to converge to $x_0\in D$ in distribution,
it is established in \cite{BC} that $(Y^k, B^k)$ converges weakly to
$(Y, B, \PP^*_x)$
in the space $C([0, \infty), \R^n\times\R^n)$ equipped with local
uniform topology.
By the Skorokhod representation theorem (see \cite{EK}, Theorem 3.1.8),
we can construct
$(Y^k, B^k)$ and $(Y, B)$ on the same probability space $(\Omega, \FF,
\overline\PP)$ so that
$(Y^k, B^k)$ converges to $(Y, B)$, $\overline\PP$-a.s., on the time
interval $[0, \infty)$
locally uniformly. Consequently,~$M^k$ converges to $M$, $\overline\PP
$-a.s., on the time interval $[0, \infty)$ locally uniformly,
where
\[
M_t=\exp\biggl( \int_0^t \bv  \,dB_s - \frac12 \int_0^t
| \bv|^2 \,ds \biggr) .
\]
Let $\PP$ be defined by $d\PP/d\overline\PP= M_t$ on $\FF_t$.
Fix $T>0$. It suffices to show that~$X^k$ converges weakly to $X$ in
the space
$C([0, T], \R^n)$. Let $\Phi$ be a continuous function on $C([0, T],
\R
^n) $ with $0\leq
\Phi\leq1$. Since $\Phi(Y^k)\to\Phi(Y)$
$\overline\PP$-a.s. and $M^k_T\to M_T$, $\overline\PP$-a.s.,
by Fatou's lemma,
%
\begin{equation}\label{e:2.6}
\E_{\overline\PP} [ \Phi(Y) M_T]
\leq\liminf_{k\to\infty}
\E_{\overline\PP} [ \Phi(Y^k) M^k_T] \leq\limsup_{k\to\infty}
\E_{\overline\PP}
[ \Phi(Y^k) M^k_T]
\end{equation}
and
%
\begin{equation}\label{e:2.7}
\E_{\overline\PP}
[(1- \Phi) (Y) M_T]\leq\liminf_{k\to\infty}
\E_{\overline\PP} [ (1-\Phi) (Y^k) M^k_T] .
\end{equation}
Summing (\ref{e:2.6}) and (\ref{e:2.7}) we obtain $\E_{\overline\PP} [M_T]
\leq
\limsup_{k\to\infty} \E_{\overline\PP} [M^k_T]$.
Note that all $M^k$'s and $M$ are continuous nonnegative
$\overline\PP$-martingales. Hence, $\E_{\overline\PP}[M^k_T]=1= \E_{\overline\PP
}[ M_T]$
and, therefore, inequalities in (\ref{e:2.6}) and (\ref{e:2.7})
are, in fact, equalities. It follows that
\[
\lim_{k\to\infty} \PP^k [ \Phi(X^k)]=
\lim_{k\to\infty} \E_{\overline\PP} [ \Phi(Y^k) M^k_T] =
\E_{\overline\PP} [ \Phi(Y) M_T] =\E_\PP[ \Phi(X) ].
\]
This proves the weak convergence of
$X^k$ under $\PP^k$ to $X$ under $\PP$.

(iii) By definition, $m(dx)=1_D(x) e^{2\bv\cdot x}$ is a
symmetrizing measure for reflecting Brownian motion $X$ with
constant drift $\bv$ on $D$.
If $m(D)<\infty$, then $\mu:=m/m(D)$
is the unique stationary distribution of $X$ on $D$.
By \cite{Fu}, Theorem 2(ii), for every bounded
$f\in\FF$,
\[
\lim_{t\to\infty} \E_x [ f(X_t)]=\lim_{t\to\infty} \E_x [ f(X^*_t)]
= h(x) \qquad\mbox{for q.e. } x\in D,
\]
where $h$ is a quasi-continuous function with $P_t h=h$ q.e.
for every $t>0$. Here $\{P_t, t\geq0\}$ is the transition
semigroup of $X^*$. Since $D$ is connected, the reflecting
Brownian motion $Y^*$ on $D$ is irreducible and so is $X^*$.
Since $m(D)<\infty$, constant $1 \in\FF$ with $\EE(1, 1)=0$.
Therefore, $X^*$ is recurrent. It follows that $h$ is constant
and equals $\int_D f(x) \mu(dx)$. Note that reflecting Brownian
motion $Y^*$ can be defined
to start from every point in $x\in
D$ and has a transition density function with respect to the
Lebesgue measure in~$D$. As $X^*$ can be obtained from $Y^*$
through Girsanov transform, the same holds for $X^*$. It
follows that for every $x\in D$,
\[
\lim_{t\to\infty} \E_x [ f(X_t)] = \lim_{s\to\infty} \E_x [
P_1f (X_s)]
= \int_D P_1 f(x) \mu(dx)= \int_D f(x) \mu(dx).
\]
Since $\FF$ is dense in the space of bounded continuous
functions on $D$, the last formula shows that $\mu$ is the
unique stationary distribution for $X^*$ and~$X$.

(iv) Since
\[
\int_{D} \exp( 2 x\cdot\bv_b) \,dx
\leq
\frac{\exp( c_0 b)} {\exp( c_0 )}
\int_{D} \exp( 2 x\cdot\bv_1) \,dx < \infty,
\]
it follows from (iii) that
${\mu}_b (dx)
=\exp( 2 x\cdot\bv_b) \,dx/\int_{D} \exp( 2 x\cdot\bv_b) \,dx $ is the
stationary probability distribution for reflected Brownian
motion in $D$ with drift $\bv_b$.

Note that for $x\in D(\eps)$, we have $\lim_{b\to\infty}
\exp(- (c_0 - \eps)b)
\exp( 2 x\cdot\bv_b) =\infty$,
so by the monotone convergence theorem
\[
\lim_{b\to\infty}
\int_{D(\eps)} \exp\bigl(- (c_0 - \eps)b\bigr)
\exp( 2 x\cdot
\bv_b) \,dx =\infty
\]
and, for similar reasons,
\[
\lim_{b\to\infty}
\int_{D \setminus D(\eps)} \exp\bigl(- (c_0 - \eps)b\bigr)
\exp( 2 x\cdot\bv_b) \,dx =0.
\]
It follows that
\begin{eqnarray*}
\lim_{b\to\infty} \mu_b \bigl(D\setminus D(\eps)\bigr)
&\leq& \lim_{b\to\infty} \frac{ \mu_b (D\setminus D(\eps))}
{\mu_b (D(\eps))} \\
&=&
\lim_{b\to\infty} \frac{\int_{D \setminus D(\eps)} \exp(- (c_0 -
\eps)b)
\exp( 2 x\cdot\bv_b) \,dx} {\int_{D(\eps)} \exp(- (c_0 - \eps)b)
\exp( 2
x\cdot
\bv_b) \,dx} \\
&=&0.
\end{eqnarray*}
Consequently, $\lim_{b\to\infty} \mu(D(\eps))= 1-
\lim_{b\to\infty} \mu(D\setminus D(\eps)) =1$.
\end{pf}

\section{Configuration space for hard core
objects}\label{sec:conf}

In this section we will start the formal presentation of our
model and we will prove two lemmas about the configuration
space.\vadjust{\goodbreak}

Suppose that $d\geq2$ and $D\subset\R^d$ is open and connected. The
set $D$ represents the space where hard
core objects may be located. Note that we did not impose any smoothness
assumptions on $\prt D$.

Consider open nonempty bounded sets $S_k \subset\R^d$, $k=1,\ldots,
N$, $N \geq1$. The sets $S_k$'s represent hard core objects.
We will think about $S_k$'s as moving or randomly placed
objects so we will use the notation $S_k(y) = S_k + y$. The
diameter of $S_k$ will be denoted by $\rho_k$.

Let $\bD'\subset\R^{Nd}$ be the set of all
${\mathbf x}= (x^1,\ldots, x^N)$, $x^k \in\R^d$, such that $S_k(x^k)
\subset D$ for all $k=1,\ldots, N$, and $S_j(x^j) \cap S_k(x^k)
=\varnothing$, for $j,k = 1,\ldots, N$, $j\ne k$.
Let $\bD\subset\R^{Nd}$ be the interior of
$\bD'$. We will call $\bD$ the configuration space.

We will prove that the configuration space $\bD$ is
connected for two
examples of $D$ and $S_k$'s. We do not aim at a great
generality because, first, the problem of characterizing $D$
and $S_k$'s such that $\bD$ is connected seems to be very hard
and, second, our main examples in Section \ref{sec:examples}
are concerned with models where connectivity of $\bD$ is rather
easy to see.
\begin{example}\label{ex:j16.2}
\begin{longlist}
\item
Suppose that there exists an upper semi-continuous function
$g\dvtx\break \R^{d-1} \to\R$ such that $D= \{(x_1,\ldots, x_d)\dvtx x_1 >
g( x_2,\ldots, x_d) \}$ and all $S_k$'s are convex
%
\begin{figure}

\includegraphics{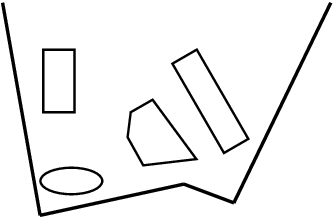}

\caption{Convex objects above the graph of a function.}
\label{fig1}
\end{figure}
(see Figure \ref{fig1}).
\item
Suppose that $D= \{(x_1,\ldots, x_d)\dvtx x_1 >0, x_2^2 +
\cdots+ x_d^2 < 1\}$ is a one-sided open cylinder, $S_k$'s are
%
\begin{figure}[t]

\includegraphics{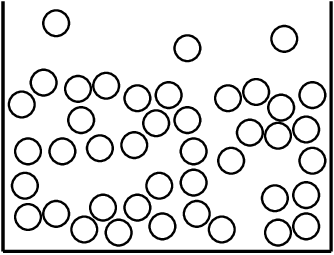}

\caption{Spherical objects in a cylindrical vessel.}
\label{fig2}
\end{figure}
open balls and $\rho_k < 1$ for $k=1,\ldots, N$
(see Figure \ref{fig2}).
\end{longlist}
\end{example}
\begin{lemma}\label{o15.2}
If $D$ and $S_k$'s are such as in Example \ref{ex:j16.2}\textup{(i)} or
\textup{(ii)}, then~$\bD$ is pathwise connected.
\end{lemma}
\begin{pf}
Suppose that ${\mathbf x},\by\in\bD$ and $\mathbf{x}\ne\by$. We will
describe a continuous motion of objects $S_k$ inside $D$ such
that the initial configuration is represented by ${\mathbf x}$ and the
terminal configuration is $\by$.

(i) Consider $\bD$ defined relative to $D$ given in Example
\ref{ex:j16.2}(i). Let $\rho_* = \max_{k=1,\ldots, N} \rho_k$.

Our argument will involve constants $c_1,c_2>0$ whose values
will be chosen later. First, we move continuously and
simultaneously all objects $S_k$ by~$c_1$ units in the
direction $(1,0,\ldots, 0)$. Let $\bz= (z^1,\ldots, z^N)$
denote the new configuration. Next we dilate the configuration
by $c_2$ units, that is, we fix $S_1$ and we move continuously
every object $S_k$, $k\ne1$, along the line segment $[z^k, z^k
+ (z^k- z^1) c_2]$ away from $S_1$ at the speed $c_2 |z^2 -
z^1|$. We move all~$S_k$, $k\ne1$, simultaneously. Note that
the objects $S_k$ will not intersect at any time because they
are convex. Let $\bz_1$ denote the configuration of the objects
at the end of the dilation.

Now we choose $c_1$ and $c_2$ so large that all objects $S_k$
are always inside $D$ and the distance between any two objects
is greater than $\rho^*$ when they are in the configuration
$\bz_1$.

Next, we start with the configuration $\by$ and we use a
similar method to move objects $S_k$ continuously from
configuration $\by$ to a configuration $\bu$, such that the
objects do not intersect in the process of moving, they always
stay inside $D$ and the distance between any two objects is
greater than $\rho^*$ when they are in the configuration $\bu$.
We make $c_1$ larger, if necessary, so that the distance from
any object in the configuration $\bz_1$ to any object in the
configuration $\bu$ is greater than $\rho_*$.

At this point, we can move all objects continuously, one by
one, from their location in configuration $\bz_1$ to their
place in configuration $\bu$. We combine motions from ${\mathbf x}$ to
$\bz_1$, then to $\bu$ and, by reversing an earlier motion,
from~$\bu$ to $\by$.

Consider the set $\Gamma$ of all points $\bz_2$ representing
the locations of objects $S_k$ at all times during the motions.
The set $\Gamma$ is connected because the motions of the
objects were continuous, $\Gamma\subset\bD$ because the
objects always stayed in~$D$ and never intersected each other
and clearly ${\mathbf x}, \by\in\Gamma$. We have proved that for any
${\mathbf x},\by\in\bD$ there exists a connected subset of $\bD$
containing both points---this proves that $\bD$ is
pathwise connected.

(ii) Now consider $\bD$ defined relative to $D$ given in
Example \ref{ex:j16.2}(ii). Recall that ${\mathbf x}= (x^1,\ldots,
x^N)$ and let $x^k=(x^k_1,\ldots, x^k_d)$ represent the center
of the $k$th ball. Find a permutation $(\pi(1),\ldots,
\pi(N))$ of $(1,\ldots, N)$ such that $x^{\pi(1)}_1 \geq
x^{\pi(2)}_1 \geq\cdots\geq x^{\pi(N)}_1$. Similarly, let $\by
= (y^1,\ldots, y^N)$ and $y^k=(y^k_1,\ldots, y^k_d)$. Let
$(\sigma(1),\ldots, \sigma(N))$ be a permutation of $(1,\ldots,
N)$ such that $y^{\sigma(1)}_1 \geq y^{\sigma(2)}_1 \geq\cdots
\geq y^{\sigma(N)}_1$. Let $b =
\max(x^{\pi(1)}_1,y^{\sigma(1)}_1) + 1$. Move the
$\pi(1)$th ball in a continuous way to a location inside $D$
such that the first coordinate of its center is equal to $b+1$
and the ball does not intersect the axis of~$D$. Moreover, we
move the ball in such a way that it does not intersect any
other ball or $\prt D$ at any time. Next, move the $\pi(2)$th
ball in a continuous way to a location inside $D$ such that the
first coordinate of its center is equal to $b+2$ and the ball
does not intersect the axis of $D$. We move the ball in such a
way that it does not intersect any other ball or $\prt D$ at
any time. Continue in this way, until we move the $\pi(N)$th
ball to a location inside $D$ such that the first coordinate of
its center is equal to $b+N$ and the ball does not intersect
the axis of $D$. We move the last ball in such a way that it
does not intersect any other ball or $\prt D$ at any time. Such
continuous motions are possible because we always take the
``top'' ball from among those remaining in the original
position and the diameter of any ball is smaller than the
radius of the cylinder $D$.

We now move the balls to the configuration $\by$ by reversing
the steps. First, we move the $\sigma(N)$th ball to the
location where its center is $y^{\sigma(N)}$, in a continuous
way, such that the ball does not intersect any other ball or
$\prt D$ at any time. Next, we move the $\sigma(N-1)$st ball
to the location where its center is $y^{\sigma(N-1)}$, in a
continuous way, such that the ball does not intersect any other
ball or $\prt D$ at any time. We continue in this way until all
balls form the configuration represented by~$\by$.

Consider the set $\Gamma$ of all points $\bz=(z^1,\ldots, z^N)$
representing the centers of all balls at all times during the
motions. The set $\Gamma$ is connected because the motions of
the balls were continuous, $\Gamma\subset\bD$ because the
balls always stayed in $D$ and never intersected each other
and we also have ${\mathbf x}, \by\in\Gamma$. We have proved that for
any ${\mathbf x},\by\in\bD$, there exists a connected subset of $\bD
$ containing both points---this proves that $\bD$ is pathwise
connected.
\end{pf}
\begin{remark}\label{o15.1}
The arguments given in the proof of Lemma \ref{o15.2}
can also be used to show that $\bD' $ is the Euclidean
closure $ \ol\bD$ of $\bD$.
We will apply this observation to tight packings,
based on the hexagonal tight packing, in proofs given later in this paper.
\end{remark}

\section{Existence of stationary distribution}\label{sec:exist}

Informally speaking, we will assume that all objects $S_k$ move
as independent reflecting Brownian motions, with drifts $(-a_k,
0,\ldots, 0)$, with $a_k>0$. Formally, the evolving system of
objects is represented by a stochastic process $\bX_t=(X^1_t,\ldots, X^N_t)$ with values in
$\ol\bD$.
In other words, the $k$th
object is represented at time $t$ by $S_k(X^k_t)$. We assume
that $\bX_t$ is $(Nd)$-dimensional reflected Brownian motion in
$\bD$ with drift
%
\begin{equation}\label{e:5.1}
\bv= (
(-a_1, 0,\ldots, 0),\ldots, (-a_N, 0,\ldots, 0) ),
\end{equation}
where $a_k >0$ for $k=1,\ldots, N$.

The $n$-dimensional volume (Lebesgue measure) of a set $A\subset\R^n$
will be denoted $m_n(A)$.
\begin{lemma}\label{lem:j16.3}
Assume that $\bD$ is connected. Let $D_b = \{x =(x_1,\ldots,
x_d)\in D\dvtx x_1 =b\}$ and $a_* = \min(a_1,\ldots, a_N)$. If
there is $b_0\in\R$ such that $D_b =\varnothing$ for all $b<b_0$
and if there is some $a<a_*$ so that
%
\begin{equation}\label{eq:j16.1}
\lim_{b\to\infty}
m_{d-1}(D_b)
\exp(-2a b) = 0,
\end{equation}
then $\bX$ has a unique stationary distribution.
\end{lemma}
\begin{pf}
The uniqueness of the stationary distribution follows from
Theorem~\ref{th:j14.2} and the fact that $\bD$ is connected.

In view of Theorem \ref{th:j14.2}(iii), it is enough to show
that $\exp( 2 {\mathbf x}\cdot\bv)$ is integrable over $\bD$. It
follows from (\ref{eq:j16.1}) that for some $c$ and all $b\geq
b_0$,
$m_{d-1}(D_b) \leq c \exp(2a b)$.
This implies that
\begin{eqnarray*}
\int_{\bD} \exp( 2 {\mathbf x}\cdot\bv) \,d{\mathbf x}
&\leq&
\int_{D^N} \exp( 2 {\mathbf x}\cdot\bv) \,d{\mathbf x}
= \prod_{k=1}^N
\int_D \exp(-2 a_k x_1) \,dx\\
&=&
\prod_{k=1}^N
\int_{b_0}^\infty
m_{d-1}(D_{x_1})
\exp(-2 a_k x_1) \,dx_1\\
&\leq&
\prod_{k=1}^N
\int_{b_0}^\infty c \exp(2a x_1) \exp(-2 a_k x_1) \,dx_1
< \infty.
\end{eqnarray*}
\upqed\end{pf}

\section{Examples of macroscopic effects}\label{sec:examples}

In this section, we will be concerned with the distribution of
the process $\bX_t$ under the stationary distribution $\mu$, so
we will suppress the time variable $t$ and we will write $\bX$
for $\bX_0$. We will also use the following notation, $\bX=
(X^1,\ldots, X^N)$ and $X^k= (X^k_1,\ldots, X^k_d)$, for
$k=1,\ldots,N$.

\subsection{Surface of a liquid}

\begin{theorem}\label{th:j18.3}
Suppose that $D$ and $S_k$'s are as in Example \ref{ex:j16.2}\textup
{(i)} or \textup{(ii)} and $D$ satisfies
(\ref{eq:j16.1})
for some $a >0$. Fix some $\alpha_k>0$ for
$k=1,\ldots, N$ and let $c_1=\inf_{{\mathbf x}\in\bD} \sum_{j=1}^N
\alpha_j x^j_1$.
Let $\lambda>0$ and $a_k = \lambda
\alpha_k$ for $k=1,\ldots, N$.
Let $\lambda_*<\infty$ be so large that $\min(a_1,\ldots, a_N) > a$ for
$\lambda\geq\lambda_*$.
Assume that $\lambda\geq\lambda_*$ and $\bX$ has
the stationary distribution which we denote $\mu_\lambda$. [Note that
$(\bX, \mu_\lambda)$
depends on $\lambda$ through drift $\bv$ in (\ref{e:5.1}).]

\begin{longlist}
\item
For any $p,\delta>0$, there exists $\lambda_0< \infty$
such that for $\lambda> \lambda_0$,
\[
\mu_\lambda(\alpha_1 X^1_1 + \cdots+ \alpha_N X^N_1
< c_1 + \delta) >1-p.
\]
\item
For any $p,\delta>0$, there exists $\lambda_0< \infty$
such that for $\lambda> \lambda_0$, with probability greater
than $1-p$, for every $k=1,\ldots, N$, for every $z\in\R^d$
with $z_1 < -\delta$, we have
%
\begin{equation}\label{eq:j22.1}
\bigl(S_k(X^k) + z\bigr) \cap
\biggl(D^c \cup\bigcup_{j\ne k} S_j(X^j) \biggr) \ne
\varnothing.
\end{equation}
\end{longlist}
\end{theorem}

Theorem \ref{th:j18.3}(i) says that if the drift of every
process $X^k$ is sufficiently large then the ``weighted center
of mass'' for a typical configuration of $S_k$'s is within an
arbitrarily small number of the infimum of weighted centers of
mass over all permissible configurations with arbitrarily
large probability.
%
\begin{figure}[b]

\includegraphics{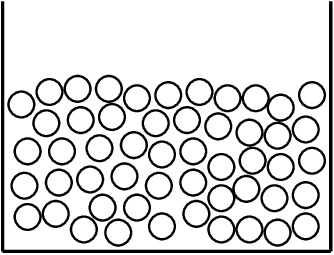}

\caption{There are no ``air bubbles'' below the surface of this
configuration.}
\label{fig3}
\end{figure}
%

%
\begin{figure}[b]

\includegraphics{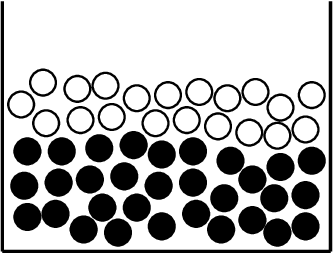}

\caption{White balls have small downward drift while black balls have large
downward drift. Every black ball is ``almost'' below every white ball.}
\label{fig4}
\end{figure}

Part (ii) of the theorem says that for an arbitrarily small
$\delta>0$, if the drift of every process $X^k$ is sufficiently
large, then with arbitrarily large probability, there is no room
in the configuration to move any object $S_k$ to a new location
that would be more than $\delta$ units in the negative
$x_1$-direction below the current location of $S_k$. This
means, in particular, that there are no spherical holes between
$S_k$'s with diameter $\max_k \rho_k$ or greater, $\delta$
units below the ``surface'' of $S_k$'s, that is, the hyperplane
$\{(x_1,\ldots,x_d) \in D\dvtx x_1
= \max_k \sup_{ y\in S_k(X^k)} y_1\}$
(see Figure \ref{fig3}).
\begin{pf*}{Proof of Theorem \ref{th:j18.3}}
(i) Recall the notation from Theorem \ref{th:j14.2}(iv).
We can identify $\bv$ in (\ref{e:5.1}) corresponding to
$\lambda=1$ with $\bv_1$ in Theorem \ref{th:j14.2}(iv).
Then
%
\begin{equation}\label{eq:j18.1}\quad
\{{\mathbf x}\in\bD\dvtx \alpha_1 x^1_1 + \cdots+ \alpha_N x^N_1
< c_1 + \delta\}
= \{{\mathbf x}\in\bD\dvtx
2{\mathbf x}\cdot\bv_1 > c_0 - 2 \delta\},
\end{equation}
where $c_0:= \sup_{{\mathbf x}\in\bD} 2{\mathbf x}\cdot\bv_1$.
It is now easy to see that part (i) of the theorem follows from
Theorem \ref{th:j14.2}(iv).

(ii) Let $\alpha_0 = \inf_{k=1,\ldots,N} \alpha_k$. Suppose that for
some configuration ${\mathbf x}\in\bD$, there exist $k$ and $z\in\R^d$
with $z_1 < -\delta$ such that (\ref{eq:j22.1}) is not satisfied.
Let~$\by$ represent the configuration which is obtained from ${\mathbf x}$ by
moving $S_k$ from~$S_k(x^k) $ to $S_k(x^k) + z$. Note that $\by\in
\bD$. Since $\by\in\bD$ and $z_1 < -\delta$, we must have $\alpha
_1 x^1_1 +\cdots+ \alpha_N x^N_1 \geq c_1 + \alpha_0 \delta$. We now
apply part (i) of the theorem to see that the family of
configurations ${\mathbf x}$, such that (\ref{eq:j22.1}) is not satisfied, has
$\mu_\lambda$-probability less than $p$, if $\lambda$ is sufficiently
large.
\end{pf*}

\subsection{Centrifuge effect}

In this example, all objects have the same shape but they are
subject to different forces (drifts $a_k$).
\begin{theorem}\label{theooo}
Suppose that $D$ and $S_k$'s are as in Example \ref{ex:j16.2}\textup
{(i)} or~\textup{(ii)}, $S_j=S_1$ for $j=2,\ldots, N$ and $D$ satisfies
(\ref{eq:j16.1}) for some $a >0$.
Let the vector~$\bv$ representing drifts be as in (\ref{e:5.1}).
Fix some $\alpha_k>0$ for
$k=1,\ldots, N$ and let $c_1=\inf_{{\mathbf x}\in\bD} \sum_{j=1}^N
\alpha_j x^j_1$.
Let $\lambda>0$ and $a_k = \lambda
\alpha_k$ for $k=1,\ldots, N$.
Let $\lambda_*<\infty$ be so large that $\min(a_1,\ldots, a_N) > a$ for
$\lambda\geq\lambda_*$.
Assume that $\lambda\geq\lambda_*$ and $\bX$ has
the stationary distribution $\mu_\lambda$.
For any $p,\delta>0$, there exists
$\lambda_0< \infty$ such that for $\lambda> \lambda_0$, with
probability greater than $1-p$, for every pair $j,k \in\{1,\ldots,
N\}$
with $\alpha_j > \alpha_k$, we have
$X^j_1 < X^k_1+\delta$.
\end{theorem}

The theorem says that in the stationary regime, with
arbitrarily large probability, if the drift is very strong, then
the identical objects $S_k$ are arranged in an almost monotone
order according to the strength of the drift
(see Figure \ref{fig4}).

\begin{pf*}{Proof of Theorem \ref{theooo}}
The idea of the proof is very similar to that of the proof of Theorem
\ref{th:j18.3}(ii).
Let $\alpha_0 = \min\{ \alpha_j-\alpha_k\dvtx  \alpha_j >
\alpha_k\}$ and note that $\alpha_0 >0$.
Suppose that for some
configuration ${\mathbf x}\in\bD$, there exist $j$ and $k$ such that
$\alpha_j > \alpha_k$ and
$x^j_1 \geq x^k_1+\delta$.
Let $\by\in\bD
$ represent the configuration which is obtained from ${\mathbf x}$ by
interchanging the positions of $S_j(x^j)$ and $S_k(x^k)$. Since $\by
\in\bD$,
$\alpha_j \geq\alpha_k + \alpha_0$ and $x^j_1 \geq x^k_1+\delta$,
we must have
\[
\sum_{i=1}^N \alpha_i x^i_1= \sum_{i=1}^N \alpha_i y^i_1
+(\alpha_j-\alpha_k)(x^j_1-x^k_1)\geq c_1 + \alpha_0 \delta.
\]
We now apply part (i) of Theorem
\ref{th:j18.3} to see that the family of configurations~${\mathbf x}$, such
that $\alpha_j > \alpha_k$ and
$x^j_1 \geq x^k_1+\delta$
for some $j$ and $k$, has $\mu_\lambda$-probability less than~$p$, if
$\lambda$ is
sufficiently large.
\end{pf*}

\subsection{Archimedes' principle}\label{sec:float}
We will discuss the phenomena of floating and sinking assuming
that $d=2$ and $S_k$'s are discs. The reason for the limited
generality of this example is that our argument is based on the
classical sphere packing problem. This problem was completely
solved in two dimensions a long time ago (see \cite{FT}) and it
also has been settled in three dimensions more recently (see
\cite{H}). The situation is more complicated in higher
dimensions (see \cite{CS,FT,H} and references therein for
details). We will further limit our discussion to the
cylindrical domain defined in Example~\ref{ex:j16.2}(ii)
because this example captures the essence of our claims.
In the following, for $x\in\R^d$ and $r>0$, $B(x, r)$
denotes the open ball with radius $r$ centered at $x$.
\begin{theorem}\label{th:j20.1}
Suppose that $d=2$,
$D=(0, \infty) \times(-1, 1)$,
$S_1=B(0,1/2)$ and $S_k= B(0,\rho)$ for $k=2,\ldots,N$, where
$\rho<1/2$. Assume that $a_2 =a_3 =\cdots= a_N$ and $\bX$ has
distribution $\mu$.

\begin{longlist}
\item
For any $p,\delta,\gamma>0$, there exists $\rho_0>0$ and
$N_0<\infty$ such that, for $\rho\in(0,\rho_0)$ and $N\geq N_0$ which
satisfy the
condition $\rho^2 N\sqrt{12}> 2 - \pi/4$, there exists $a_0>0$
such that if $a_2\geq a_0$ and $a_1/a_2 := \gamma_1 \leq
(\pi/(4\sqrt{12})-\gamma)/ \rho^2$ then
\[
\mu\Bigl( X^1_1 < \max_{k=2,\ldots, N} X^k_1 -1/2-\delta
\Bigr) < p.
\]

\item For any $p,\delta,\gamma>0$, there exists $\rho_0>0$ and
$N_0<\infty$ such that, for $\rho\in(0,\rho_0)$ and $N\geq N_0$ which
satisfy the
condition $\rho^2 N\sqrt{12}> 2 - \pi/4$, there exists $a_0>0$
such that if $a_2\geq a_0$ and $a_1/a_2 := \gamma_2 \geq
(\pi/(4\sqrt{12})+\gamma)/ \rho^2$ then
\[
\mu( X^1_1 > 1/2 + \delta) < p.
\]
\end{longlist}
\end{theorem}

The theorem is a form of Archimedes' principle. The first part
of our result says that if the drift of the large ball is
smaller than the sum of the drifts of displaced small balls
then the large ball will ``float,'' that is, its uppermost
point will be at least very close to the ``surface'' of the
``liquid'' (or above the surface). The second part says that if
the drift of the large ball is greater than the sum of the
drifts of displaced small balls, then the large ball will sink
to the bottom. Both results assume that the system is in the
stationary distribution and all drifts are large
(see Figure \ref{fig5}).

\begin{figure}

\includegraphics{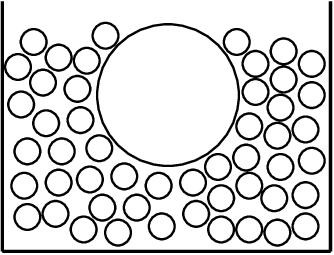}

\caption{The large disc is floating because its drift is less
than the sum of drifts of displaced small discs (properly scaled).}
\label{fig5}
\end{figure}

The condition $\rho^2 N\sqrt{12} > 2 - \pi/4$ is needed to make sure
that there is an ample
supply of small discs to make the large disc float.
\begin{pf*}{Proof of Theorem \ref{th:j20.1}}
(i) We will use some results about disc packing in the plane from
\cite{CS,FT,H}. The usual honeycomb lattice packing of disjoint discs
has density $\pi/\sqrt{12}$ and this is the highest possible disc
packing density.

Consider the unique honeycomb packing ${\mathcal S}$ of
open discs
with radii $\rho$
in the whole plane in which some adjacent discs have their centers on
the line parallel to the first axis and one disc is centered at 0.
For a set $A$, we will say that $S_1,\ldots, S_k$ is a \textit{honeycomb
disc packing in $A$} if it contains all discs in ${\mathcal S}$ that are
contained in~$A$.

By abuse of notation, we will use $| \cdot |$ to denote the
area of a planar set and also the cardinality of a finite set.

Consider arbitrary $p,\gamma,\delta>0$. We fix some $\beta> 2 - \pi/4$
and assume that $2 - \pi/4 < \rho^2 N\sqrt{12}< \beta$.
Note that it is sufficient to prove the theorem for every
fixed~$\beta$.

Suppose that for some configuration ${\mathbf x}\in\bD$, we have
%
\begin{equation}\label{eq:o9.10}
x^1_1 < \max_{k=2,\ldots, N} x^k_1 -1/2-\delta.
\end{equation}
Let $\alpha_1 = \gamma_1$ and $\alpha_k = 1$ for $2\leq k \leq
N$. Let $\lambda= a_2$. Then $\alpha_1 x^1_1 +\cdots+
\alpha_N x^N_1 = \gamma_1 x^1_1 + x^2_1 +\cdots+ x^N_1$. In
view of Theorem \ref{th:j18.3}(i), it will suffice to show
that for small $\rho>0$ there exists $\delta_1
=\delta_1(\rho,N)>0$ such that for some $\by\in\bD$,
%
\begin{equation}\label{eq:j19.2}
\gamma_1 x^1_1 + x^2_1 +\cdots+ x^N_1
> \gamma_1 y^1_1 + y^2_1 +\cdots+ y^N_1 +\delta_1.
\end{equation}

In view of Remark \ref{o15.1}, it will suffice to show that (\ref{eq:j19.2})
is satisfied for some $\by\in\bD' =\ol\bD$.

We divide the next part of the proof into cases and ``subcases.''

\textit{Case $1$}.
If there exist $z=(z_1,z_2)$ and $2 \leq k \leq N$ such that
$z_1 < - \rho$ and
%
\begin{equation}\label{eq:o11.1}
\bigl(S_k(x^k) + z\bigr) \cap
\Biggl(D^c \cup\bigcup_{j\ne k} S_j(x^j) \Biggr)
= \varnothing,
\end{equation}
then we choose the smallest $k$ with this property and let
\[
\by= (y^1,\ldots, y^N)
=(x^1, x^2 ,\ldots, x^{k-1}, x^k + z , x^{k+1},\ldots,
x^N).
\]
Then (\ref{eq:j19.2}) holds with $\delta_1 = \rho$.

\textit{Case $2$}.
In this case, we suppose
that there are no $z$ and $k$ satisfying (\ref{eq:o11.1}).
Informally speaking, this implies that the disc configuration
represented by~${\mathbf x}$ has a density bounded below by some absolute
constant $c_1>0$. We will now make this assertion precise. Consider a
square
\begin{eqnarray*}
Q &=& Q(r_1,r_2, \rho) \\
&=& \{(z_1,z_2) \in\R^2\dvtx r_1-3\rho< z_1 < r_1 +
3 \rho, r_2-3\rho< z_2 < r_2 + 3 \rho\}
\end{eqnarray*}
and assume that $Q\subset D$ and
$r_1 + 3 \rho\leq\max_{k=2,\ldots, N} x^k_1$.
We have assumed that there are no $z$ and $k$ satisfying
(\ref{eq:o11.1}) so $B((r_1 ,r_2) , \rho)$ must intersect at least
one disc $S_k(x^k)$ with $2\leq k \leq N$. It follows that $S_k(x^k)
\subset Q$ and, therefore, the area of $(\bigcup_{2\leq k \leq
N} S_k(x^k) ) \cap Q$ is greater than or equal to $c_1 := \pi/
36$ times the area of $Q$.

We have assumed that $\rho^2 N\sqrt{12}< \beta$ so $ \max_{k=2,\ldots, N} x^k_1 < \beta_1$
for some
$\beta_1 = \beta_1(\beta)< \infty$.

Let $D_1 = (0,2 \beta_1) \times(-1,1)$ and $D_2 = D_1
\setminus S_1(x^1)$.
For $0<b<2\beta_1$, define
$D_2^b := \{(x_1,x_2) \in D_2\dvtx x_1 <
b\}$. An upper\vspace*{1pt} estimate of the length of $\prt D_2^b$ is $c_2
:= \pi+ 4 + 4\beta_1$. Let $D^b_3= D^b_3(\rho) = \{x \in
D^b_2\dvtx \dist(x , \prt D^b_2)< 2\rho\}$ and $D^b_4= D^b_4(\rho) = \{x
\in(D^b_2)^c\dvtx \dist(x , \prt D^b_2)< 2\rho\}$.
We have $|D_3^b| \leq4\rho c_2 $ and $|D_4^b| \leq4\rho c_2 $ for
small~$\rho$.
Hence, the number $N_1 = N_1(b)$ of discs of radius $\rho$ in the
honeycomb packing in $D^b_2$ satisfies
\[
(|D_2^b| - 4\rho c_2 ) /\bigl( \rho^2 \sqrt{12}\bigr)
\leq N_1
\leq(|D_2^b| + 4\rho c_2) /\bigl( \rho^2 \sqrt{12}\bigr).
\]
We cannot pack $N_1$ discs in $D_2^{b_1}$
if
$|D^{b_1}_2| \leq
|D^b_2| - 9\rho c_2$. If $\rho$ is small then this condition follows
from $b - b_1 \geq9\rho c_2$.
So in any
configuration of $N_1$ discs in $D_2$, the $N_1$th disc
from the bottom
will be at most $9\rho c_2$ units below the position of the
$N_1$th disc from the bottom in the honeycomb packing of $D_2$.

If $b_2 = b + \rho\sqrt{3}$, then the honeycomb packings of $D^b_2$
and $D^{b_2}_2$ differ by one row of discs, or a part of one row. The
centers of the discs in the top row of $D^{b_2}_2$ are $\rho\sqrt{3}$
units above the centers of discs in the top row of $D^b_2$. There are
no more than $1/\rho$ discs in the top row of $D^{b_2}_2$. Consider
any $N_3$ such that $N_1(b) \leq N_3 \leq N_1(b_2)$. For any
configuration of $N_3$ discs in $D_2$, the $N_3$th particle from the
bottom will be at most $9\rho c_2+ \rho\sqrt{3}$ units below the
position of the $N_3$th particle from the bottom in the honeycomb
packing of~$D^{b_2}_2$.

Let $S_2(z^2),\ldots, S_{N}(z^{N})$ be the disc configuration
in $D_2$ obtained by taking \mbox{$N-1$} discs in the honeycomb
packing of $D_2$ with the lowest first coordinates.
We label $z^k$'s so that $z^j_1 \leq z^k_1$ if $j < k$.
Then
%
\begin{equation}\label{e:5.6}
x^k_1 \geq z^k_1 - \bigl(9 c_2+ \sqrt{3}\bigr)\rho \qquad\mbox{for }2\leq k
\leq N.
\end{equation}
Therefore, for any $2\leq N_2 \leq N$,
%
\begin{equation}\label{o12.5}
x^2_1 +\cdots+ x^{N_2}_1
\geq z^2_1 +\cdots+ z^{N_2}_1 - (N_2-1)\bigl(9 c_2+ \sqrt{3}\bigr)\rho.
\end{equation}

Recall that we assume that there do not exist $z$ and $k$ satisfying
(\ref{eq:o11.1}).

\textit{Case} 2(a).
Suppose that $\max_{k=2,\ldots, N} x^k_1 >
\max_{k=2,\ldots, N} z^k_1 + \delta/2$.
Let $\K_1 = \{2\leq k \leq N\dvtx x^k_1 > \max_{j=2,\ldots, N} z^j_1 +
\delta/4\}$.
Note that $|\K_1| \geq
c_1\delta/(8\pi\rho^2)$
for
small $\rho$ because $x^k$'s represent a configuration with density
bounded below by $c_1$. We have for small $\rho$,
\[
\sum_{k\in\K_1} x^k_1 - \sum_{k\in\K_1} z^k_1
\geq|\K_1|   \delta/4
\geq c_1\delta^2/(32\pi\rho^2) .
\]
Let $N_2 = N - |\K_1|$. We apply (\ref{o12.5}) to the discs
corresponding to $1\leq k \leq N$, $k\notin\K_1$ to obtain
\[
\sum_{2\leq k \leq N, k\notin\K_1} x^k_1
- \sum_{2\leq k \leq N, k\notin\K_1} z^k_1
\geq- (N_2-1)\bigl(9 c_2+ \sqrt{3}\bigr)\rho.
\]
Combining both estimates, we see that, for small $\rho$,
\begin{eqnarray*}
\sum_{2\leq k \leq N} x^k_1
- \sum_{2\leq k \leq N} z^k_1
&\geq& c_1\delta^2/
(32\pi\rho^2)
-(N_2-1)\bigl(9 c_2+ \sqrt{3}\bigr)\rho\\
&\geq&
c_1\delta^2/ (32\pi\rho^2) -
\bigl(\beta/\bigl(\rho^2 \sqrt{12}\bigr)-1\bigr) \bigl(9 c_2+ \sqrt{3}\bigr)\rho\\
&\geq&
c_1\delta^2/ (64\pi\rho^2) .
\end{eqnarray*}
Hence, for small $\rho>0$, (\ref{eq:j19.2}) holds with $y^1 =
x^1$, $y^k = z^k$ for $2\leq k \leq N$ and $\delta_1 =
c_1\delta^2/
(64\pi\rho^2)$.
Note that
$\by\in\ol\bD$
because $S_2(z^2),\ldots, S_{N}(z^{N})$
is a part of the honeycomb\vspace*{1pt} packing of $D_2$, so these discs are
disjoint and they are disjoint with $S_1(x^1)$.

\textit{Case} 2(b).
Next suppose that
%
\begin{equation}\label{o13.3}
\max_{k=2,\ldots, N} x^k_1 \leq\max_{k=2,\ldots, N} z^k_1 + \delta/2.
\end{equation}
Recall $\gamma$ from the statement of the theorem. We can
assume without loss of generality that $\gamma\in(0,1)$. If
$\rho$ is small then we can find $y^1_1 \in(x^1_1 +
\delta(1-\gamma)/4 , x^1_1 + \delta/4)$ such that the line
$M := \{(u_1,u_2)\dvtx u_1 = (x^1_1+y^1_1)/2\}$
is a line of symmetry for the honeycomb packing ${\mathcal S}$. Let $y^1 =
(y^1_1, x^1_2)$. Note that $S_1(y^1) \setminus S_1(x^1)$ is
``filled'' with the discs from the family $S_2(z^2),\ldots,
S_{N}(z^{N})$ when $\rho$ is small because, in view of
(\ref{eq:o9.10}) and (\ref{o13.3}),
\[
\max_{k=2,\ldots, N} z^k_1 \geq\max_{k=2,\ldots, N} x^k_1 - \delta/2
> x^1_1 + 1/2 + \delta- \delta/2
=( y^1_1 + \delta/4 + 1/2) +\delta/4.
\]
Let $\M\dvtx \R^2 \to\R^2$ be the symmetry with respect to\vspace*{1pt} $M$.
Let $\K_2 = \{2\leq k \leq N\dvtx S_k(z^k) \cap S_1(y^1) \ne
\varnothing\}$. For $k\in\K_2$, let $y^k = \M(z^k)$.
For all
other $k$, let $y^k = z^k$.
Since $\{S_2(z^2),\ldots, S_{N}(z^{N})\}$
is a part of the honeycomb packing of $D_2=D_1\setminus S_1(x^1)$,
$\{S_2(y^2),\ldots, S_{N}(y^{N})\}$
is a part of the honeycomb packing of $\R^2\setminus S_1(y^1)$.
Since any disc in ${\mathcal S}$ that intersects
$S_1(x_1)$ has to be in
the positive half space $\{(\xi_1, \xi_2)\in\R^2\dvtx \xi_1>0\}$,
$\{S_2(y^2),\ldots, S_{N}(y^{N})\}$
is in\vspace*{1pt} fact a~part of the honeycomb packing of $D_1\setminus S_1(y^1)$.
Therefore,
$\by\in\ol\bD$.

The next part of our argument is best explained using physical
intuition. Suppose that all discs $S_k$ are made of material
with mass density 1 and they are in gravitational field with
constant acceleration 1 in the negative direction along the
first axis. When we move disc $S_1(x^1)$ to the new position at
$S_1(y^1)$, then we do $(y^1_1 - x^1_1)\pi/4$ units of work,
which is at least $(\delta(1-\gamma)/4)\pi/4$. We can imagine
that the mass in $S_1(x^1) \cap S_1(y^1)$ does not move and we
only move the mass in $S_1(y^1) \setminus S_1(x^1)$ to its
symmetric image $S_1(x^1)\setminus S_1(y^1)$ under $\M$. When
$\rho$ is very small, the discs $S_k(z^k)$, $k\in\K_2$, have
total mass arbitrarily close to $(\pi/\sqrt{12}) |S_1(y^1)
\setminus S_1(x^1)|$, uniformly spread over $S_1(y^1) \setminus
S_1(x^1)$. Hence,\vspace*{1pt} the amount of work needed to move all discs
$S_k(z^k)$ to $S_k(y^k)$ for $k\in\K_2$ is negative and
smaller than $-(\pi/\sqrt{12})(\delta(1-2\gamma)/4)\pi/4$ for
small $\rho$. In other words,
\[
\sum_{k\in\K_2} (y^k_1 - z^k_1) \pi\rho^2
\leq-\bigl(\pi/\sqrt{12}\bigr)\bigl(\delta(1-2\gamma)/4\bigr)\pi/4 .
\]

Recall the assumption that $\gamma_1 \leq
(\pi/(4\sqrt{12})-\gamma)/ \rho^2$. We have $y^1_1 \leq x^1_1 +
\delta/4$ so
\[
\gamma_1 (y^1_1 - x^1_1) \leq\gamma_1 \delta/4
\leq\bigl(\bigl(\pi/\bigl(4\sqrt{12}\bigr)-\gamma\bigr)/ \rho^2\bigr) \delta/4.
\]
Combining the last two estimates we obtain
%
\begin{eqnarray} \label{o14.1}\quad
&&\gamma_1 (y^1_1 - x^1_1) +
\sum_{k\in\K_2} (y^k_1 - z^k_1) \nonumber\\
&&\qquad\leq-\bigl(\pi/\sqrt{12}\bigr)\bigl(\delta(1-2\gamma)/4\bigr)(\pi/4)/(\pi\rho^2)
+\bigl(\bigl(\pi/\bigl(4\sqrt{12}\bigr)-\gamma\bigr)/ \rho^2\bigr) \delta/4\\
&&\qquad\leq-\gamma\delta/(12 \rho^2).\nonumber
\end{eqnarray}
Recall that $N\leq\beta/(\rho^2 \sqrt{12})$
and let $\delta_1 = \gamma\delta/(48 \rho^2)$. We apply (\ref{o12.5})
and~(\ref{o14.1}) to see that for small $\rho>0$,
\begin{eqnarray*}
&&\gamma_1 x^1_1 + x^2_1 +\cdots+ x^N_1\\[2pt]
&&\qquad\geq\gamma_1 x^1_1 + z^2_1 +\cdots+ z^N_1 -(N-1)\bigl(9 c_2 + \sqrt{3}\bigr)
\rho\\[2pt]
&&\qquad\geq\gamma_1 x^1_1 + z^2_1 +\cdots+ z^N_1 -\bigl(9 c_2 + \sqrt{3}\bigr) \rho
\beta/\bigl(\rho^2 \sqrt{12}\bigr)\\[2pt]
&&\qquad= \gamma_1 x^1_1 + \sum_{k\in\K_2} z^k_1
+ \sum_{2\leq k \leq N, k\notin\K_2}
z^k_1 -\beta\bigl(9 c_2 + \sqrt{3}\bigr) /\bigl(\rho\sqrt{12}\bigr)\\[2pt]
&&\qquad= \gamma_1 x^1_1 + \sum_{k\in\K_2} z^k_1
+ \sum_{2\leq k \leq N, k\notin\K_2}
y^k_1 -\beta\bigl(9 c_2 + \sqrt{3}\bigr) /\bigl(\rho\sqrt{12}\bigr)\\[2pt]
&&\qquad> \gamma_1 y^1_1 + \sum_{k\in\K_2} y^k_1 + 2\delta_1
+ \sum_{2\leq k \leq N, k\notin\K_2}
y^k_1 -\beta\bigl(9 c_2 + \sqrt{3}\bigr) /\bigl(\rho\sqrt{12}\bigr)\\[2pt]
&&\qquad= \gamma_1 y^1_1 + y^2_1 +\cdots+ y^N_1 +\gamma\delta/(24 \rho^2)
-\beta\bigl(9 c_2 + \sqrt{3}\bigr) /\bigl(\rho\sqrt{12}\bigr)\\
&&\qquad> \gamma_1 y^1_1 + y^2_1 +\cdots+ y^N_1 + \delta_1.
\end{eqnarray*}
Hence, condition (\ref{eq:j19.2}) is satisfied. This completes
the proof of part (i) of the theorem.

(ii) The second part of the theorem can be proved just like the
first part. The proof is identical up to (\ref{o13.3}). In the
part\vspace*{1pt} following (\ref{o13.3}), all we have to do is to take
$y^1_1 \in(x^1_1 - \delta/4 , x^1_1 - \delta(1-\gamma)/4)$
instead of $y^1_1 \in(x^1_1 + \delta(1-\gamma)/4 , x^1_1 +
\delta/4)$, because in this part we want to move $S_1(x^1)$
down, not up. We leave the details to the reader.
\end{pf*}

\subsection{Inert objects}

We will model inertia of objects $S_k$ by changing the rules of
reflection. When two different objects $S_j$ and $S_k$ reflect
from each other, they will no longer receive the same amount of
push to keep them apart. One way to formalize this idea is to
say that when the process $\bX$ hits the boundary of $\bD$ at
the time when $S_j$ hits $S_k$, then $\bX$ is not reflected
normally but it is subject to oblique reflection. The drift
$a_k$ will not be assumed to be related to the value of inertia
for $S_k$. In other words, $a_k$'s may model forces which have
strength dependent on factors other than the inertial mass.

We will assume that the standard deviation for oscillations of
$S_k$ is inversely proportional to the inertia of $S_k$. The
reason for this assumption is purely technical. The assumption
allows us to transform the problem to the model covered by
Theorem \ref{th:j14.2}. In general, it is not easy to find an
explicit formula for the stationary distribution of reflected
Brownian motion with oblique reflection (even if the process
has no drift).

Next, we formalize the ideas stated above. Let $m_k>0$ be the
parameter representing the inertia for $S_k$. Let
\begin{eqnarray*}
\T({\mathbf x}) &=& (m_1 x^1,\ldots, m_N x^N),  \qquad {\mathbf x}\in\R
^{Nd},\\
\wt\bD&=& \T(\bD),\qquad
\wt\bX_t = \T(\bX)_t .
\end{eqnarray*}
We assume that $\wt\bX$ is reflected Brownian motion in $\wt
\bD$ (with the normal reflection), with drift
\[
\wt\bv= (
(-a_1 m_1, 0,\ldots, 0),\ldots, (-a_N m_N, 0,\ldots, 0) ).
\]
Let $\wt\mu$ denote the stationary distribution for $\wt\bX$
and let $\mu'$ be the corresponding stationary distribution for
$\bX$. Note that under $\mu'$, the quadratic variation process
for $X^k$ is $t/m_k^2$.

In the present example, if two objects $S_j$ and $S_k$ reflect
from each other, then the infinitesimal displacement of $S_j$ is
$m_k/m_j$ times the infinitesimal displacement of~$S_k$.

We will illustrate the effect of inertia using the same model
as in the Section~\ref{sec:float}.
\begin{theorem}\label{theor}
Suppose that $d=2$, $D$ is as in Example \ref{ex:j16.2}\textup{(ii)},
$S_1=B(0$, $1/2)$ and $S_k= B(0,\rho)$ for $k=2,\ldots,N$, where
$\rho<1/2$. Assume that $a_2 =a_3 =\cdots= a_N$, $m_1>1$,
$m_2=\cdots= m_N = 1$ and $\bX$ has distribution $\mu'$.

\begin{longlist}
\item
For any $p,\delta,\gamma>0$, there exists $\rho_0>0$ and
$N_0<\infty$ such that, for $\rho\in(0,\rho_0)$ and $N\geq N_0$ which
satisfy the
condition $\rho^2 N\sqrt{12}> 2 - \pi/4$, there exists $a_0>0$
such that if $a_2\geq a_0$ and $a_1/a_2 := \gamma_1 \leq
(\pi/(4\sqrt{12})-\gamma)/ (\rho^2 m_1)$, then
\[
\mu' \Bigl( X^1_1 < \max_{k=2,\ldots, N} X^k_1 -1/2-\delta
\Bigr) < p.
\]

\item
For any $p,\delta,\gamma>0$, there exists $\rho_0>0$ and
$N_0<\infty$ such that, for $\rho\in(0,\rho_0)$ and $N\geq N_0$ which
satisfy the
condition $\rho^2 N\sqrt{12}> 2 - \pi/4$, there exists $a_0>0$
such that if $a_2\geq a_0$ and $a_1/a_2 := \gamma_2 \geq
(\pi/(4\sqrt{12})+\gamma)/ (\rho^2m_1)$, then
\[
\mu' ( X^1_1 > 1/2 + \delta) < p.
\]
\end{longlist}
\end{theorem}

The theorem says that the higher is inertia $m_1 $, the lower is
the critical drift $a_1$ that makes the disc $S_1$ sink. We
note parenthetically that behavior of real particulate matter
can be paradoxical, unlike in our example. Large and heavy
particles may move to the top of a mixture of small and large
particles under some circumstances (see \cite{RSPS}).
\begin{pf*}{Proof of Theorem \ref{theor}}
We can use the same reasoning as in the proof of Theorem
\ref{th:j20.1} but with a twist. Theorem \ref{th:j18.3}(i)
must be applied to the process $\wt\bX$ under $\wt\mu$, so we
have to analyze $m_1 a_1 x^1_1 + a_2 x^2_1 +\cdots+ a_N x^N_1
$ rather than $ a_1 x^1_1 + a_2 x^2_1 +\cdots+ a_N x^N_1 $.
Hence, $\gamma_1$ in (\ref{eq:j19.2}) must have an extra factor
of $1/m_1$. The constant $\gamma_1$ in the present theorem has
that extra factor, as compared to the constant $\gamma_1$ in
Theorem \ref{th:j20.1}. With this change, the proof of Theorem
\ref{th:j20.1} applies in the present context.
\end{pf*}

\section*{Acknowledgments}
We are grateful to Amir Dembo, Persi Diaconis, Joel Lebowitz,
Charles Radin, Benedetto Scoppola and Jason Swanson for very useful advice.
We thank the anonymous referees for helpful comments.



%
\printaddresses


\begin{thebibliography}{26}

\bibitem{BB}
\begin{barticle}[mr]
\bauthor{\bsnm{Bass},~\bfnm{Richard~F.}\binits{R.~F.}} \AND
  \bauthor{\bsnm{Burdzy},~\bfnm{Krzysztof}\binits{K.}}
(\byear{2008}).
\btitle{On pathwise uniqueness for reflecting {B}rownian motion in
  {$C\sp{1+\gamma}$} domains}.
\bjournal{Ann. Probab.}
\bvolume{36}
\bpages{2311--2331}.
\bid{doi={10.1214/08-AOP390}, issn={0091-1798}, mr={2478684}}
\end{barticle}
\endbibitem

\bibitem{BBC}
\begin{barticle}[mr]
\bauthor{\bsnm{Bass},~\bfnm{Richard~F.}\binits{R.~F.}},
  \bauthor{\bsnm{Burdzy},~\bfnm{Krzysztof}\binits{K.}} \AND
  \bauthor{\bsnm{Chen},~\bfnm{Zhen-Qing}\binits{Z.-Q.}}
(\byear{2005}).
\btitle{Uniqueness for reflecting {B}rownian motion in lip domains}.
\bjournal{Ann. Inst. H. Poincar\'e Probab. Statist.}
\bvolume{41}
\bpages{197--235}.
\bid{doi={10.1016/j.anihpb.2004.06.001}, issn={0246-0203}, mr={2124641}}
\end{barticle}
\endbibitem

\bibitem{BC}
\begin{barticle}[mr]
\bauthor{\bsnm{Burdzy},~\bfnm{Krzysztof}\binits{K.}} \AND
  \bauthor{\bsnm{Chen},~\bfnm{Zen-Qing}\binits{Z.-Q.}}
(\byear{1998}).
\btitle{Weak convergence of reflecting {B}rownian motions}.
\bjournal{Electron. Comm. Probab.}
\bvolume{3}
\bpages{29--33 (electronic)}.
\bid{issn={1083-589X}, mr={1625707}}
\end{barticle}
\endbibitem

\bibitem{BPS}
\begin{barticle}[auto:STB|2011-03-03|12:04:44]
\bauthor{\bsnm{Burdzy},~\bfnm{K.}\binits{K.}},
  \bauthor{\bsnm{Pal},~\bfnm{S.}\binits{S.}} \AND
  \bauthor{\bsnm{Swanson},~\bfnm{J.}\binits{J.}}
(\byear{2010}).
\btitle{Crowding of Brownian spheres}.
\bjournal{ALEA Lat. Am. J. Probab. Math. Stat.}
\bvolume{7}
\bpages{193--205}.
\end{barticle}
\endbibitem

\bibitem{C1}
\begin{barticle}[mr]
\bauthor{\bsnm{Chen},~\bfnm{Zhen~Qing}\binits{Z.~Q.}}
(\byear{1993}).
\btitle{On reflecting diffusion processes and {S}korokhod decompositions}.
\bjournal{Probab. Theory Related Fields}
\bvolume{94}
\bpages{281--315}.
\bid{doi={10.1007/BF01199246}, issn={0178-8051}, mr={1198650}}
\end{barticle}
\endbibitem

\bibitem{C2}
\begin{barticle}[mr]
\bauthor{\bsnm{Chen},~\bfnm{Zhen-Qing}\binits{Z.-Q.}}
(\byear{1996}).
\btitle{Reflecting {B}rownian motions and a deletion result for {S}obolev
  spaces of order {$(1,2)$}}.
\bjournal{Potential Anal.}
\bvolume{5}
\bpages{383--401}.
\bid{doi={10.1007/BF00275474}, issn={0926-2601}, mr={1401073}}
\end{barticle}
\endbibitem

\bibitem{CFKZ}
\begin{barticle}[mr]
\bauthor{\bsnm{Chen},~\bfnm{Z.~Q.}\binits{Z.~Q.}},
  \bauthor{\bsnm{Fitzsimmons},~\bfnm{P.~J.}\binits{P.~J.}},
  \bauthor{\bsnm{Kuwae},~\bfnm{K.}\binits{K.}} \AND
  \bauthor{\bsnm{Zhang},~\bfnm{T.~S.}\binits{T.~S.}}
(\byear{2008}).
\btitle{Perturbation of symmetric {M}arkov processes}.
\bjournal{Probab. Theory Related Fields}
\bvolume{140}
\bpages{239--275}.
\bid{doi={10.1007/s00440-007-0065-2}, issn={0178-8051}, mr={2357677}}
\end{barticle}
\endbibitem

\bibitem{CFTYZ}
\begin{barticle}[mr]
\bauthor{\bsnm{Chen},~\bfnm{Z.~Q.}\binits{Z.~Q.}},
  \bauthor{\bsnm{Fitzsimmons},~\bfnm{P.~J.}\binits{P.~J.}},
  \bauthor{\bsnm{Takeda},~\bfnm{M.}\binits{M.}},
  \bauthor{\bsnm{Ying},~\bfnm{J.}\binits{J.}} \AND
  \bauthor{\bsnm{Zhang},~\bfnm{T.~S.}\binits{T.~S.}}
(\byear{2004}).
\btitle{Absolute continuity of symmetric {M}arkov processes}.
\bjournal{Ann. Probab.}
\bvolume{32}
\bpages{2067--2098}.
\bid{doi={10.1214/009117904000000432}, issn={0091-1798}, mr={2073186}}
\end{barticle}
\endbibitem

\bibitem{CS}
\begin{bbook}[mr]
\bauthor{\bsnm{Conway},~\bfnm{J.~H.}\binits{J.~H.}} \AND
  \bauthor{\bsnm{Sloane},~\bfnm{N.~J.~A.}\binits{N.~J.~A.}}
(\byear{1999}).
\btitle{Sphere Packings, Lattices and Groups},
\bedition{3rd} ed.
\bseries{Grundlehren der Mathematischen Wissenschaften [Fundamental Principles
  of Mathematical Sciences]}
\bvolume{290}.
\bpublisher{Springer}, \baddress{New York}.
\bid{mr={1662447}}
\end{bbook}
\endbibitem

\bibitem{DLM}
\begin{bmisc}[mr]
\bauthor{\bsnm{Diaconis},~\bfnm{Persi}\binits{P.}},
\bauthor{\bsnm{Lebeau},~\bfnm{Gilles}\binits{G.}} \AND
\bauthor{\bsnm{Michel},~\bfnm{L.}\binits{L.}}
(\byear{2009}).
\bhowpublished{Geometric analysis for the
Metropolis algorithm on Lipschitz domains.
\textit{Inventiones Mathematicae}
\href{http://dx.doi.org/10.1007/s00222-010-0303-6}{DOI:10.1007/s00222-010-0303-6}.}
\end{bmisc}
\endbibitem

\bibitem{DM}
\begin{barticle}[mr]
\bauthor{\bsnm{Dieker},~\bfnm{A.~B.}\binits{A.~B.}} \AND
  \bauthor{\bsnm{Moriarty},~\bfnm{J.}\binits{J.}}
(\byear{2009}).
\btitle{Reflected {B}rownian motion in a wedge: Sum-of-exponential stationary
  densities}.
\bjournal{Electron. Comm. Probab.}
\bvolume{14}
\bpages{1--16}.
\bid{issn={1083-589X}, mr={2472171}}
\end{barticle}
\endbibitem

\bibitem{DR}
\begin{barticle}[mr]
\bauthor{\bsnm{Dupuis},~\bfnm{Paul}\binits{P.}} \AND
  \bauthor{\bsnm{Ramanan},~\bfnm{Kavita}\binits{K.}}
(\byear{2002}).
\btitle{A time-reversed representation for the tail probabilities of stationary
  reflected {B}rownian motion}.
\bjournal{Stochastic Process. Appl.}
\bvolume{98}
\bpages{253--287}.
\bid{doi={10.1016/S0304-4149(01)00151-X}, issn={0304-4149}, mr={1887536}}
\end{barticle}
\endbibitem

\bibitem{EK}
\begin{bbook}[mr]
\bauthor{\bsnm{Ethier},~\bfnm{Stewart~N.}\binits{S.~N.}} \AND
  \bauthor{\bsnm{Kurtz},~\bfnm{Thomas~G.}\binits{T.~G.}}
(\byear{1986}).
\btitle{Markov Processes: Characterization and Convergence}.
\bpublisher{Wiley}, \baddress{New York}.
\bid{doi={10.1002/9780470316658}, mr={0838085}}
\end{bbook}
\endbibitem

\bibitem{FT}
\begin{bbook}[mr]
\bauthor{\bsnm{Fejes~T{\'o}th},~\bfnm{L.}\binits{L.}}
(\byear{1953}).
\btitle{Lagerungen in der {E}bene, Auf der {K}ugel und Im {R}aum}.
\bseries{Die Grundlehren der Mathematischen Wissenschaften in
  Einzeldarstellungen Mit Besonderer Ber\"ucksichtigung der Anwendungsgebiete
  Band}
\bvolume{LXV}.
\bpublisher{Springer}, \baddress{Berlin}.
\bid{mr={0057566}}
\end{bbook}
\endbibitem

\bibitem{F}
\begin{barticle}[mr]
\bauthor{\bsnm{Fradon},~\bfnm{Myriam}\binits{M.}}
(\byear{2010}).
\btitle{Brownian dynamics of globules}.
\bjournal{Electron. J. Probab.}
\bvolume{15}
\bpages{142--161}.
\bid{issn={1083-6489}, mr={2594875}}
\bptnote{check year}
\end{barticle}
\endbibitem

\bibitem{Fu}
\begin{bincollection}[mr]
\bauthor{\bsnm{Fukushima},~\bfnm{Masatoshi}\binits{M.}}
(\byear{1983}).
\btitle{Capacitary maximal inequalities and an ergodic theorem}.
In \bbooktitle{Probability Theory and Mathematical Statistics ({T}bilisi,
  1982)}.
\bseries{Lecture Notes in Math.}
\bvolume{1021}
\bpages{130--136}.
\bpublisher{Springer}, \baddress{Berlin}.
\bid{doi={10.1007/BFb0072909}, mr={0735979}}
\end{bincollection}
\endbibitem

\bibitem{FOT}
\begin{bbook}[mr]
\bauthor{\bsnm{Fukushima},~\bfnm{Masatoshi}\binits{M.}},
  \bauthor{\bsnm{{\=O}shima},~\bfnm{Y{\=o}ichi}\binits{Y.}} \AND
  \bauthor{\bsnm{Takeda},~\bfnm{Masayoshi}\binits{M.}}
(\byear{1994}).
\btitle{Dirichlet Forms and Symmetric {M}arkov Processes}.
\bseries{de Gruyter Studies in Mathematics}
\bvolume{19}.
\bpublisher{de Gruyter}, \baddress{Berlin}.
\bid{mr={1303354}}
\end{bbook}
\endbibitem

\bibitem{H}
\begin{barticle}[mr]
\bauthor{\bsnm{Hales},~\bfnm{Thomas~C.}\binits{T.~C.}}
(\byear{2006}).
\btitle{Historical overview of the {K}epler conjecture}.
\bjournal{Discrete Comput. Geom.}
\bvolume{36}
\bpages{5--20}.
\bid{doi={10.1007/s00454-005-1210-2}, issn={0179-5376}, mr={2229657}}
\end{barticle}
\endbibitem

\bibitem{HW}
\begin{barticle}[mr]
\bauthor{\bsnm{Harrison},~\bfnm{J.~M.}\binits{J.~M.}} \AND
  \bauthor{\bsnm{Williams},~\bfnm{R.~J.}\binits{R.~J.}}
(\byear{1987}).
\btitle{Multidimensional reflected {B}rownian motions having exponential
  stationary distributions}.
\bjournal{Ann. Probab.}
\bvolume{15}
\bpages{115--137}.
\bid{issn={0091-1798}, mr={0877593}}
\end{barticle}
\endbibitem

\bibitem{LS}
\begin{barticle}[mr]
\bauthor{\bsnm{Lions},~\bfnm{P.~L.}\binits{P.~L.}} \AND
  \bauthor{\bsnm{Sznitman},~\bfnm{A.~S.}\binits{A.~S.}}
(\byear{1984}).
\btitle{Stochastic differential equations with reflecting boundary conditions}.
\bjournal{Comm. Pure Appl. Math.}
\bvolume{37}
\bpages{511--537}.
\bid{doi={10.1002/cpa.3160370408}, issn={0010-3640}, mr={0745330}}
\end{barticle}
\endbibitem

\bibitem{MR}
\begin{bbook}[mr]
\bauthor{\bsnm{Ma},~\bfnm{Zhi~Ming}\binits{Z.~M.}} \AND
  \bauthor{\bsnm{R{\"o}ckner},~\bfnm{Michael}\binits{M.}}
(\byear{1992}).
\btitle{Introduction to the Theory of (nonsymmetric) {D}irichlet Forms}.
\bpublisher{Springer}, \baddress{Berlin}.
\bid{mr={1214375}}
\end{bbook}
\endbibitem

\bibitem{O}
\begin{barticle}[mr]
\bauthor{\bsnm{Osada},~\bfnm{Hirofumi}\binits{H.}}
(\byear{1996}).
\btitle{Dirichlet form approach to infinite-dimensional {W}iener processes with
  singular interactions}.
\bjournal{Comm. Math. Phys.}
\bvolume{176}
\bpages{117--131}.
\bid{issn={0010-3616}, mr={1372820}}
\end{barticle}
\endbibitem

\bibitem{R}
\begin{bmisc}[auto:STB|2011-03-03|12:04:44]
\bauthor{\bsnm{Radin},~\bfnm{C.}\binits{C.}}
\bhowpublished{The ``most probable'' sphere packings, and models of soft
  matter. Univ. Texas. Available at
  \url{http://www.ma.utexas.edu/users/radin/spheres.html}.}
\end{bmisc}
\endbibitem

\bibitem{RSPS}
\begin{barticle}[mr]
\bauthor{\bsnm{Rosato},~\bfnm{Anthony}\binits{A.}},
  \bauthor{\bsnm{Strandburg},~\bfnm{Katherine~J.}\binits{K.~J.}},
  \bauthor{\bsnm{Prinz},~\bfnm{Friedrich}\binits{F.}} \AND
  \bauthor{\bsnm{Swendsen},~\bfnm{Robert~H.}\binits{R.~H.}}
(\byear{1987}).
\btitle{Why the {B}razil nuts are on top: Size segregation of particulate
  matter by shaking}.
\bjournal{Phys. Rev. Lett.}
\bvolume{58}
\bpages{1038--1040}.
\bid{doi={10.1103/PhysRevLett.58.1038}, issn={0031-9007}, mr={0879723}}
\end{barticle}
\endbibitem

\bibitem{RDXRC}
\begin{barticle}[auto:STB|2011-03-03|12:04:44]
\bauthor{\bsnm{Rutgers},~\bfnm{M.~A.}\binits{M.~A.}},
  \bauthor{\bsnm{Dunsmuir},~\bfnm{J.~H.}\binits{J.~H.}},
  \bauthor{\bsnm{Xue},~\bfnm{J.~Z.}\binits{J.~Z.}},
  \bauthor{\bsnm{Russel},~\bfnm{W.~B.}\binits{W.~B.}} \AND
  \bauthor{\bsnm{Chaikin},~\bfnm{P.~M.}\binits{P.~M.}}
(\byear{1996}).
\btitle{Measurement of the hard-sphere equation of state using screened charged
  polystyrene colloids}.
\bjournal{Phys. Rev. B}
\bvolume{53}
\bpages{5043--5046}.
\end{barticle}
\endbibitem

\end{thebibliography}
\end{document}